\documentclass{birkjour}

\usepackage{graphicx}
\usepackage[ansinew]{inputenc}    
\usepackage{fancyhdr}
\usepackage{easybmat}
\usepackage{color}
\usepackage{verbatim}
\usepackage{amsmath,amstext,amssymb,}
\usepackage{mathrsfs,amscd}

\newtheorem{thm}{Theorem}[section]
\newtheorem{corollary}[thm]{Corollary}
\newtheorem{lemma}[thm]{Lemma}
\newtheorem{proposition}[thm]{Proposition}

\renewcommand{\v}{{\bf{v}}}
\newcommand{\y}{{\bf{y}}}
\newcommand{\0}{{\bf{0}}}

\begin{document}

\title[ Envelope of Mid-Hyperplanes of a Hypersurface]{ \LARGE{Envelope of Mid-Hyperplanes of a Hypersurface}}

\author[A.Cambraia Jr. ]{Ady Cambraia Jr.}
\address{
Departamento de Matem\'atica, UFV, Vi\c cosa, Minas Gerais, Brazil}

\email{ady.cambraia@ufv.br}
\author[M. Craizer]{Marcos Craizer}
\address{Departamento de Matem\'atica-PUC-Rio, Rio de Janeiro, Brazil.}
\email{craizer@puc-rio.br}

\begin{abstract}
Given $2$ points of a smooth hypersurface, their mid-hyperplane is the hyperplane passing through their mid-point and the intersection of their tangent 
spaces. In this paper we study the envelope of these mid-hyperplanes (EMH) at pairs whose tangent spaces are transversal. We prove that 
this envelope consists of centers of conics having contact of order at least $3$ with the hypersurface at both points. Moreover, 
we describe general conditions for the EMH to be a smooth hypersurface. These results are extensions of the corresponding well-known results for curves. 
In the case of curves, if the EMH is contained in a straight line, the curve is necessarily affinely symmetric with respect to the line. 
We show through a counter-example that this property does not hold for hypersurfaces.
\end{abstract}

\thanks{The second author thanks CNPq for financial support during the preparation of this paper.}



\date{February 14, 2017}

\maketitle

\section{Introduction}
\label{intro}

Given a pair of points in a smooth convex planar curve, its mid-line is the line that passes through its mid-point and the intersection of the corres\-ponding tangent lines. If these tangent lines are parallel, the mid-line is the line through $M$ parallel to both tangents. When both points coincide, the mid-line is just the affine normal at the point. The envelope of mid-lines is an important affine invariant symmetry set associated with the curve. It is important in computer graphics and has been studied by many authors (\cite{Sapiro},\cite{Giblin},\cite{Giblin2},\cite{Holtom},\cite{Warder}). The envelope of mid-lines of planar curves can be divided into $3$ parts: The Affine Envelope Symmetry Set (AESS), corresponding to pairs with non-parallel tangent lines, the Mid-Points Parallel Tangent Locus (MPTL), corresponding to pairs with parallel tangent lines, and the Affine Evolute, corresponding to coincident points. 

The concept of mid-line has a quite natural generalization to a hypersurface $S$ in the affine $(N+1)$-space: For a pair $(p_1,p_2)$ in $S$, the {\it mid-hyperplane} is the affine hyperplane that passes through the mid-point $M$ of $(p_1,p_2)$ and the intersection of the tangent spaces at $p_1$ and $p_2$ (note that this intersection is a co-dimension $2$ affine space). It is then natural to ask what is the structure of the envelope of mid-hyperplanes. In this paper, we study this set assuming that the tangent spaces at $p_1$ and $p_2$ are transversal (for $N=1$, this set is the AESS). 
We shall call {\it Envelope of Mid-Hyperplanes} (EMH) the envelope of mid-hyperspaces of pairs $(p_1,p_2)$ with transversal tangent spaces. 
The envelope of mid-planes (N=2) with parallel tangent planes is called MPTS and has been studied in \cite{Warder}. The envelope of mid-planes (N=2) corresponding to coincident points is called Affine Mid-Planes Evolute and has been studied in \cite{Cambraia}.

The AESS is very well studied and coincides with the locus of center of conics having contact of order $\geq 3$ with the curve at $2$ points. Moreover,
if both contacts are of order exactly $3$, the AESS is regular (\cite{Giblin},\cite{Holtom}). 
In this article we prove that each point of the EMH is the center of a conic having contact of order $\geq 3$ with the hypersurface at $2$ points. 
Moreover, we describe a general conditition for the regularity of the EMH that generalizes the re\-gu\-larity condition for curves. 
This general condition is algebraic, but we give geometric interpretations in some particular cases.

The reflection property of the AESS is very significant for symmetry recognition: If the AESS is contained in a straight line $l$, then the curve itself is invariant under an affine reflection with axis $l$ (\cite{Giblin},\cite{Holtom}). Unfortunately, the reflection property does not extend to the EMH. 
We give an example where the EMH is contained in a plane (N=2) but the surface $S$ is not invariant under an affine reflection. An interesting
question here is to understand which kind of symmetry is implied by the inclusion of the EMH in a hyperplane. 

This work is part of the doctoral thesis of the first author under the supervision of the second author.


\section{Review of affine differential geometry of hypersurfaces}


\hspace{0.35cm} In this section we review some basic concepts of affine differential geometry of hypersurfaces in $(N+1)$-space (for details, see \cite{Nomizu}). Denote by $D$ the canonical connection and by
$[....]$ the standard volume form in the affine $(N+1)$-space.
Let $S$ be a hypersurface and denote by $\mathfrak{X}(S)$ the tangent bundle of $S$. Given a transversal vector field $\xi$, write the Gauss equation
\begin{equation}\label{eq:Gauss}
D_XY=\nabla_XY + h(X,Y)\xi,
\end{equation}
$X, Y \in \mathfrak{X}(S)$, where $h$ is a symmetric bilinear form and $\nabla$ is a torsion free connection in $S$. We shall assume that $h$ is non-degenerate, which is independent of the choice of $\xi$.
The volume form induces a volume form in $S$ by the relation
\begin{equation*}
\theta(X_1,...,X_N)=\left[ X_1,....,X_N,\xi \right].
\end{equation*}
The metric $h$ also defines a volume form in $S$: Given $X_i \in \mathfrak{X}(S)$, $1\leq i\leq N$,  let
\begin{equation*}
\theta_h(X_1,...,X_N):=|\det(h(X_i,X_j))|^\frac{1}{2}.
\end{equation*}
Next theorem is fundamental in affine differential geometry (\cite{Nomizu}, ch.II):
\begin{thm}
There exists, up to signal, a unique transversal vector field $\xi$ such that
$\nabla\theta=0$ and $\theta=\theta_h$. The vector field $\xi$ is called the {\it affine normal vector field} and the corresponding metric $h$ the {\it Blaschke metric} of the surface.
\end{thm}


For $p\in S$, let $\nu(p)$ be the linear functional in $(N+1)$-space such that
\begin{equation*}\label{eq:DefineConormal}
\nu(p)(\xi)=1 \ \ \ \ \textrm{and} \ \ \ \ \nu(p)(X)=0 \ \forall \ X \in
T_pS.
\end{equation*}
The differentiable map $\nu$
is called the {\it conormal map}. It satisfies the following property (\cite{Nomizu}, ch.II):

\begin{proposition} \label{propconormal}
Let $S$ be a non-degenerate hypersurface and  $\nu$ the conormal map. Then
\begin{equation*}
D_Y\nu(\xi)=0 \ \ \textrm{and} \ \ D_Y\nu(X)=-h(Y,X), \ \ \forall \ X,
Y \in \mathfrak{X}(S).
\end{equation*}
\end{proposition}

\begin{corollary}\label{propconormalqq}
If $X\in \mathfrak{X}(\mathbb{R}^{N+1})$ is any vector field, then
$$
D_Y\nu(X)=-h(Y,X^T), \ \ \ Y \in \mathfrak{X}(S),
$$
where $X=X^T+\lambda\xi, \lambda \in \mathbb{R}$ and $X^T$ is the tangent component of $X$.
\end{corollary}
\begin{proof}
We have
$$
D_Y\nu(X)=D_Y\nu(X^T+\lambda\zeta)=D_Y\nu(X^T)+\lambda D_Y\nu(\zeta)=-h(Y,X^T),
$$
thus proving the corollary.
\end{proof}

\begin{lemma}\label{lemma:MixedTerm}
Assume that $S$ is the graph of $f(x,\y)$, $\y=(y_1,...,y_{N-1})$, i.e., 
\begin{equation}
\psi(x,\y)=\left( x, \y, f(x,\y) \right)
\end{equation}
is a parameterization of $S$. 
Then, at any point $(x,\y)$, $h(\psi_x,\psi_{y_j})=0$ if  and only if $f_{xy_j}=0$.
\end{lemma}
\begin{proof}
Observe first that, from equation \eqref{eq:Gauss}, we obtain
$$
\left[\psi_x,\psi_{\y},\psi_{xy_j}\right]=h(\psi_x,\psi_{y_j}) \left[ \psi_x,\psi_{\y},\xi \right].
$$
On the other hand, a direct calculation shows that
$$
[\psi_x,\psi_{\bf y},\psi_{xy_j}]=f_{xy_j}.
$$
Since $\left[ \psi_x,\psi_{\y},\xi \right]\neq 0$, we conclude the lemma.
\end{proof}


\section{Envelope of Mid-Hyperplanes}

Let $S$ be a non-degenerate convex hypersurface. Take points  $p_1, p_2 \in S$ and let $S_1\subset S$ and $S_2\subset S$ be open subsets around $p_1$ and $p_2$, respectively. Denote $h_1$ and $h_2$ the Blaschke metrics of $S_1$ and $S_2$, respectively. We shall assume
that the tangent spaces at points of $S_1$ are transversal to tangent spaces at points of $S_2$.

\subsection{Basic definitions}

Denote by $M(p_1,p_2)$ the mid-point and by $C(p_1,p_2)$ the mid-chord of $p_1$ and $p_2$, i.e., 
$$
M(p_1,p_2)=\dfrac{p_1+p_2}{2},\ \ C(p_1,p_2)=\dfrac{p_1-p_2}{2}.
$$ 
The mid-hyperplane of $(p_1,p_2)$ is the affine hyperplane that contains $M(p_1,p_2)$ and the intersection $Z$ of the tangent spaces at $p_1$ and $p_2$.
Let $F:S_1 \times S_2 \times \mathbb{R}^{N+1} \longrightarrow \mathbb{R}$ be given by
\begin{equation}\label{eq111}
F(p_1,p_2,X)=\left(\nu_2(C)\nu_1+\nu_1(C)\nu_2\right)(X-M).
\end{equation}
where $\nu_i$ is the co-normal map of $S_i$. 
It is not difficult to verify that, for $(p_1,p_2)$ fixed, $F(p_1,p_2,X)=0$ is the equation of the mid-hyperplane.

Consider frames  $\{Z_1, ...., Z_{N-1}\}$ of $Z$, each $Z_j$ being smooth functions of $(p_1,p_2)\in S_1\times S_2$. 
Consider also vector fields $Y_1,Y_2$ 
such that $Y_i(p_1,p_2)$ is tangent to $S_i$ and $h_i$-orthogonal to $Z$. 

We want to find $X$ satisfying $F=F_{p_1}=F_{p_2}=0$, for some $p_1\in S_1$, $p_2\in S_2$. Since $Y_i$ and $Z_j$ are $h_i$-orthogonals, 
$\{Y_i,Z_1,..,Z_{N-1}\}$ is a basis of $T_{p_i}S_i$, $i=1, 2$. Thus we have to find $X$ in the following system:
\begin{equation}\label{epm}
\left\{
  \begin{array}{ll}
    F(p_1,p_2,X)=0 \\
    F_{p_1}(p_1,p_2,X)(Y_1)=0 \\
    F_{p_2}(p_1,p_2,X)(Y_2)=0 \\
    F_{p_1}(p_1,p_2,X)(Z_j)=0 \\
    F_{p_2}(p_1,p_2,X)(Z_j)=0.
  \end{array}
\right.
\end{equation}
The notation $F_{p_i}(p_1,p_2,X)(W)$ corresponds to the partial derivative of $F$ with respect $p_i$ in the direction $W \in T_{p_i}S_i$, thus keeping $p_j$, $j\neq i$ and $X$ fixed.

\subsection{Solutions of the system \eqref{epm}}\label{sec:mainresult}

We begin with the following simple lemma:

\begin{lemma}\label{derivadasconormais}
We have that
$$
D_{Y_1}\nu_1=a\nu_1+b\nu_2 \ \ \ \textrm{and} \ \ \
D_{Y_2}\nu_2=\bar{a}\nu_1+\bar{b}\nu_2,
$$
where $a, b, \bar{a}, \bar{b}$ are given by
$$
a= -\dfrac{h_1(Y_1,X_2)}{\nu_1(X_2)},\ \  b=-\dfrac{h_1(Y_1,X_1)}{\nu_2(X_1)},\ \
\bar{a}=-\dfrac{h_2(Y_2,X_2)}{\nu_1(X_2)},\ \ \bar{b}=-\dfrac{h_2(Y_2,X_1)}{\nu_2(X_1)},
$$
for any $X_1\in T_{p_1}S_1$, $X_2\in T_{p_2}S_2$.
\end{lemma}

\begin{proof} 
Take a basis $\{\nu_1,\nu_2,\zeta_1,...,\zeta_{N-1}\}$ of the dual space $\mathbb{R}_{N+1}$. Thus we can write the linear functional $D_{Y_1}\nu_1$ as a linear combination of the
basis vector, i.e., $D_{Y_1}\nu_1= a\nu_1+b\nu_2+\sum_{j=1}^{N-1}c_j\zeta_j$.
Since $D_{Y_1}\nu_1(Z_j)=-h_1(Y_1,Z_j)=0$ we obtain $c_j=0$ and so $D_{Y_1}\nu_1= a\nu_1+b\nu_2$.  
Applying $D_{Y_1}\nu_1$ to any tangent vector field $X_1$ on $S_1$ we get
$D_{Y_1}\nu_1(X_1)=b\nu_2(X_1)$, thus proving the formula for $b$. The other formulas are proved similarly.
\end{proof}

\begin{proposition} \label{teoprincipal}
The first three equations of the system \eqref{epm} admit a solution if and only if
\begin{equation}\label{eq:nu1nu2}
\nu_1(C)=-\lambda\nu_2(C),
\end{equation}
where
\begin{equation}\label{eq:Lambda}
\lambda=\left(\displaystyle\frac{\nu_1^2(Y_2)}{\nu_2^2(Y_1)}\displaystyle\frac{h_1(Y_1,Y_1)}{h_2(Y_2,Y_2)}\right)^{1/3}.
\end{equation}
\end{proposition}
\begin{proof}
Since $\nu_1(Y_1)=\nu_2(Y_2)=0$, it follows that the derivative
$F_{p_1}(Y_1)$ is given by
$$
D_{Y_1}\nu_1(C)\nu_2(X-M)+\nu_2(C)D_{Y_1}\nu_1(X-M)+\dfrac{1}{2}\nu_2(Y_1)\nu_1(X-M)-\dfrac{1}{2}\nu_1(C)\nu_2(Y_1),
$$
By Lemma \ref{derivadasconormais} we obtain
\begin{equation*}
F_{p_1}(Y_1)=\left(2b\nu_2(C)\nu_2 +\dfrac{1}{2}\nu_2(Y_1)\nu_1\right)(X-M) -\dfrac{1}{2}\nu_1(C)\nu_2(Y_1) + aF.
\end{equation*}
Similarly
\begin{equation*}
F_{p_2}(Y_2)= \left(2\bar{a}\nu_1(C)\nu_1 -\dfrac{1}{2}\nu_1(Y_2)\nu_2\right)(X-M) -\dfrac{1}{2}\nu_2(C)\nu_1(Y_2)+ \bar{b}F.
\end{equation*}
Using that $F=0$, the equations $F_{p_1}(Y_1)=0$ and $F_{p_2}(Y_2)=0$ can be simplified to
\begin{equation}\label{eq3}
\left(-\dfrac{\nu_2(Y_1)\nu_1(C)\nu_2}{\nu_2(C)}+2b\nu_2(C)\nu_2\right)(X-M)=\dfrac{1}{2}\nu_1(C)\nu_2(Y_1),
\end{equation}
\begin{equation*}\label{eq4}
\left(-\dfrac{2\bar{a}\nu_1^2(C)\nu_2}{\nu_2(C)}-\nu_1(Y_2)\nu_2\right)(X-M)=\dfrac{1}{2}\nu_2(C)\nu_1(Y_2).
\end{equation*}
These equations, after some simple calculations, leads to
\begin{equation*}\label{eq5}
b\nu_2^3(C)\nu_1(Y_2)=-\bar{a}\nu_1^3(C)\nu_2(Y_1),
\end{equation*}
which, together with lemma \ref{derivadasconormais}, proves the proposition.
\end{proof}

Next lemma is a consequence of the first three equations of system \eqref{epm}:

\begin{lemma}\label{corda}
From equation \eqref{eq:nu1nu2}, we can write 
\begin{equation}\label{eq:C}
C= A\left(Y_1-\dfrac{\lambda\nu_2(Y_1)}{\nu_1(Y_2)}Y_2\right)+\sum_{j=1}^{N-1} \alpha_j Z_j,
\end{equation}
for some $A\in\mathbb{R}$, $\alpha_j\in\mathbb{R}$. 
Then
\begin{equation}\label{eq:X-M}
X-M=B\left(Y_1+\dfrac{\lambda\nu_2(Y_1)}{\nu_1(Y_2)}Y_2\right)+\sum_{j=1}^{N-1} \beta_j Z_j,
\end{equation}
where $\beta_j \in \mathbb{R}$ and 
\begin{equation}\label{eq:B}
B=-\dfrac{\lambda A}{\lambda+4Ab}.
\end{equation}
\end{lemma}

\begin{proof}
It follows from $F=0$ and equation \eqref{eq:nu1nu2} that $\nu_1(X-M)=\lambda\nu_2(X-M)$. Then equation \eqref{eq:X-M} holds, for some $B\in\mathbb{R}$.
From equation \eqref{eq3} we have
$$
\nu_2(X-M)=\dfrac{-\lambda\nu_2(C)\nu_2(Y_1)}{\lambda\nu_2(Y_1)+4b\nu_2(C)}.
$$
We conclude that 
$$
B=-\dfrac{\lambda\nu_2(C)}{\lambda\nu_2(Y_1)+4b\nu_2(C)}=-\dfrac{\lambda A}{\lambda+4Ab},
$$
thus proving the lemma.
\end{proof}

Next theorem is the main result of the section and says that the ge\-o\-me\-try of the EMH occurs in the plane generated by $Y_1$ and $Y_2$ (see figure \ref{fig1}).

\begin{figure}[ht]
  \centering
  \includegraphics[scale=0.45]{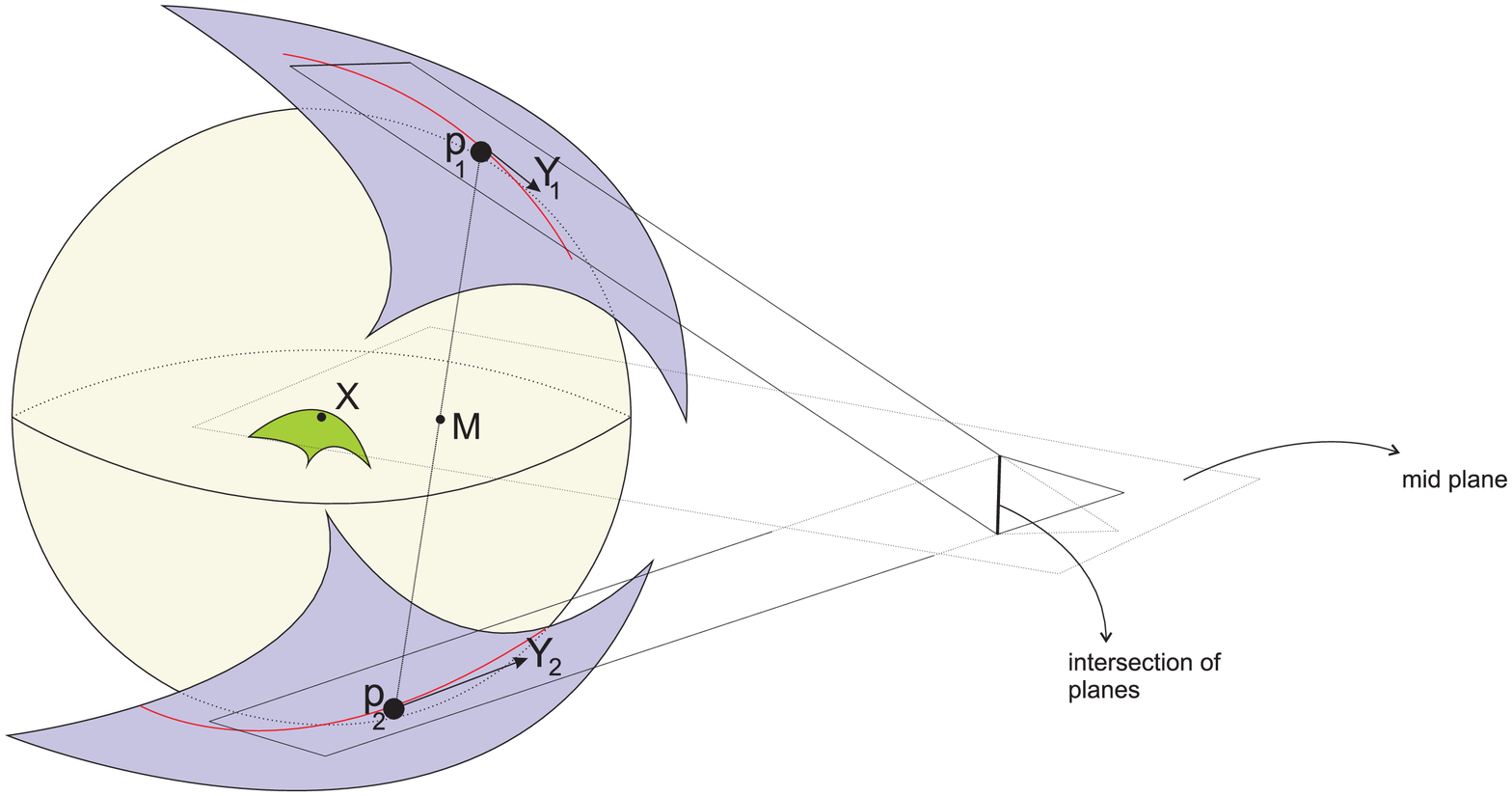}\\
  \caption{The geometry of the EMH.}
  \label{fig1}
\end{figure}

\begin{thm}\label{segundaparte}
The system \eqref{epm} admits a solution if and only if  $C,Y_1,$ and $Y_2$ are co-planar and 
equation \eqref{eq:nu1nu2} holds. Moreover, the solution of the system is given by
\begin{equation}
X-M=B\left(Y_1+\dfrac{\lambda\nu_2(Y_1)}{\nu_1(Y_2)}Y_2\right),
\end{equation}
where $\lambda$ and $B$ are given by equations \eqref{eq:Lambda} and \eqref{eq:B}, respectively.
\end{thm}

\begin{proof}
We must show that $\alpha_j=\beta_j=0$ at equations \eqref{eq:C} and \eqref{eq:X-M}, respectively.  For this, we shall consider the last two equations of the system \eqref{epm}.
The derivative $F_{p_1}(Z_j)$ is given by
$$
F_{p_1}(Z_j)= D_{Z_j}\nu_1(C)\nu_2(X-M)+\nu_2(C)D_{Z_j}\nu_1(X-M) .
$$
Thus, from Corollary \ref{propconormalqq}, 
\begin{equation}\label{eq:Fp1}
F_{p_1}(Z_j)=-h_1(Z_j,C^T)\nu_2(X-M)- \nu_2(C)h_1(Z_j,(X-M)^T),
\end{equation}
where $V^T$ denotes the projection of $V$ in $T_{p_1}S_1$ along the direction of the affine normal of $S_1$ at $p_1$. 
From equations \eqref{eq:C} and \eqref{eq:X-M} we obtain
$$
C^T= A\left(Y_1-\dfrac{\lambda\nu_2(Y_1)}{\nu_1(Y_2)}Y_2^T\right) +\sum_{k=1}^{N-1}\alpha_k Z_k
$$
and
$$
(X-M)^T=B\left(Y_1+\dfrac{\lambda\nu_2(Y_1)}{\nu_1(Y_2)}Y_2^T\right) +\sum_{k=1}^{N-1} \beta_k Z_k.
$$
Substituting these equations in equation \eqref{eq:Fp1} we obtain 
\begin{equation*}
F_{p_1}(Z_j)=-\nu_2(Y_1) \sum_{k=1}^{N-1}(B\alpha_k+A\beta_k)h_1(Z_j,Z_k).
\end{equation*}
Similarly we obtain
$$
F_{p_2}(Z_j)=\lambda\nu_2(Y_1)\sum_{k=1}^{N-1}(-\alpha_k B  + \beta_k A)h_2(Z_j,Z_k).
$$
Since $S_1$ and $S_2$ are convex, $(h_i(Z_j,Z_k))$, $i=1,2$, are definite matrices, hence non-degenerate. Moreover, non-parallel tangent planes imply that  $\nu_2(Y_1)$ and $\lambda$ are non-zero. 
Thus equations $F_{p_1}(Z_j)=F_{p_2}(Z_j)=0$ are equivalent to $\alpha_j B  + \beta_j A=-\alpha_j B  + \beta_j A=0$, which implies that $\alpha_j=\beta_j=0$.
\end{proof}


\section{Conics with $3+3$ contact with the surface}


Given two non-degenerate locally convex hypersurfaces $S_1$ and $S_2$, consider conics that makes contact of order $\geq 3$ with $S_i$
at points $p_i$, $i=1,2$, in directions $Y_i$ which are $h_i$-orthogonals to the intersection $Z$ of $T_{p_1}S_1$ and $T_{p_2}S_2$. 
We shall prove in this section that the set of centers of these $3+3$ conics coincides with the set EMH.

Along this section, we shall assume that $S_i$ is the graph of a function $f_i(u_i,\v_i)$, $i=1,2$, $\v_i=(v_{i,1},...,v_{i,N-1})$, and consider the normal vectors
\begin{equation}\label{eq:Normali}
N_i=\left( -(f_i)_{u_i}, ...,-(f_i)_{v_{i,j}},...., 1      \right)
\end{equation}
to $S_i$. Let $F:S_1\times S_2\times\mathbb{R}^3\to\mathbb{R}$  given by
\begin{equation}\label{eq:Normais1}
F(p_1,p_2,X)=\left( (N_1\cdot C)N_2+(N_2\cdot C)N_1 \right)\cdot(X-M),
\end{equation}
where $\cdot$ denotes the canonical inner product,  $N_i$ is given by equation \eqref{eq:Normali}, $M$ is the mid-point of $(p_1,p_2)$ and $C$ is the mid-chord of $(p_1,p_2)$. Then $F=0$ is the equation of the mid-plane of $(p_1,p_2)$.

\begin{lemma}\label{modelolocalsuperficies1}
Assume that the pair $(p_1,p_2)$ generates a point of EMH. Then by an affine change of coordinates,
we may assume that $p_1=(0,\0,1)$, $p_2=(0,\0,-1)$ and $S_1$ and $S_2$ are graphs of
\begin{equation}\label{eq:parameterS1}
f_1(u_1,\v_1)=1+\epsilon pu_1-\dfrac{1}{2} (p^2+\epsilon)u_1^2+\sum_{j_1,j_2} a_{j_1,j_2}v_{1,j_1}v_{1,j_2}+O(3) 
\end{equation}
and
\begin{equation}\label{eq:parameterS2}
f_2(u_2,\v_2)=-1-\epsilon pu_2+\dfrac{1}{2} (p^2+\epsilon)u_2^2+\sum_{j_1,j_2}b_{j_1,j_2}v_{2,j_1}v_{2,j_2}+O(3),
\end{equation}
where $(p,\0,0)$ is the corresponding point in EMH, $\epsilon=\pm 1$, $\epsilon p<0$, $A=(a_{j_1,j_2})$ and $B=(b_{j_1,j_2})$ are positive or negative definite. As a consequence, $(p,\0,0)$ is the center of a  conic making contact of order $\geq 3$
with $S_i$ at $p_i$ in the direction $Y_i$ $h_i$-orthogonal to the affine space $Z=span\{ \frac{\partial}{\partial \v} \}$. If $\epsilon=1$, the conic is an ellipse, while
if $\epsilon=-1$, the conic is a hyperbola.
\end{lemma}

\begin{proof}
Consider $p_1 \in S_1$ and $ p_2 \in S_2$ with non-parallel tangent planes. By an adequate affine change of variables, we may assume that
$p_1=(0,\0,1)$, $p_2=(0,\0,-1)$ and the mid-plane of $(p_1,p_2)$ is $z=0$. Since, by theorem \ref{segundaparte}, $Y_1,Y_2$ and $(0,\0,1)$ are co-planar, we may also assume that $Y_1$ and $Y_2$ 
are in the $xz$-plane.  We may also assume that theaffine space $Z$, intersection of the tangent planes $T_{p_1}S_1$ and $T_{p_2}S_2$ is the $\y$-space.
By lemma \ref{lemma:MixedTerm}, these conditions implies that the coefficients of $u_1v_{1,j}$ and $u_2v_{2,j}$ are zero. Since the tangent plane at $p_1$
contains $Z$, $S_1$ is the graph of a function $f_1$ of the form 
\begin{equation*}\label{eq:parameterS1d}
f_1(u_1,\v_1)=1+\epsilon pu_1-\dfrac{1}{2}(p^2+\epsilon)u_1^2+\sum_{j_1,j_2}a_{j_1,j_2}v_{1,j_1}v_{1,j_2}+O(3),
\end{equation*}
for $\epsilon=\pm 1$, $\epsilon p<0$. 
The tangent plane to $S_2$ at $(0,\0,-1)$ is the reflection of the tangent plane to $S_1$ at $(0,\0,1)$, so $S_2$ is the graph of a function $f_2$
of the form
\begin{equation*}\label{eq:parameterS2d}
f_2(u_2,\v_2)=-1-\epsilon pu_2+\dfrac{\delta}{2}(p^2+\epsilon)u_2^2+\sum_{j_1,j_2}b_{j_1,j_2}v_{2,j_1}v_{2,j_2}+O(3),
\end{equation*}
for some $\delta\in\mathbb{R}$.
From these formulas we obtain 
$$
N_1=(-\epsilon p+(p^2+\epsilon)u_1, -2A\v_1, 1),\ \ N_2=(\epsilon p-(p^2+\epsilon)u_2, -2B\v_2, 1).
$$
Using 
$$
C=(u_1-u_2,\v_1-\v_2,f_1-f_2),\ \ M=\frac{1}{2}(u_1+u_2,\v_1+\v_2,f_1+f_2),
$$
equation \eqref{eq:Normais1} leads to $F=2z$ at $u_1=\v_1=u_2=\v_2=0$. Straightforward but long calculations also leads 
to $F_{u_1}=x+pz-p$, $F_{u_2}=(p^2-p^2\delta-\delta)x+pz+p$, $F_{\v_1}=A\y$, $F_{\v_2}=B\y$ at $u_1=\v_1=u_2=\v_2=0$.
Thus the system $F=F_{u_1}=F_{\v_1}=F_{u_2}=F_{\v_2}=0$ at the origin becomes
\begin{equation}\label{eq:SystemEPM1}
\left\{
\begin{array}{ll}
-2z=0 \\
x+pz-p=0 \\
A\y=0 \\
(p^2-p^2\delta-\delta)x+pz+p=0 \\
B\y=0
\end{array}
\right. \ .
\end{equation}
Since, by hypothesis, this system admit a solution, this solution must be $(p,\0,0)$. Thus we conclude that $\delta=1$.
It is not difficult to verify now that there exists a conic centered at $(p,\0,0)$ contained in the plane $xz$ and making contact of order
$\geq 3$ with $S_1$ at $p_1$ and $S_2$ at $p_2$. Moreover, this conic is an ellipse if $p<0$ and a hyperbola if $p>0$.
\end{proof}

\begin{lemma}\label{modelolocalsuperficies2}
Consider a conic that makes contact of order $\geq 3$ with $S_i$ at points $p_i$, $i=1,2$, in directions $Y_i$ which are $h_i$-orthogonals to the intersection $Z$ of $T_{p_1}S_1$ and $T_{p_2}S_2$. Then, by an affine change of coordinates,
we may assume that $p_1=(0,\0,1)$, $p_2=(0,\0,-1)$ and $S_1$ and $S_2$ are graphs of functions $f_1$ and $f_2$ given by equations \eqref{eq:parameterS1} and \eqref{eq:parameterS2},
where $(p,\0,0)$ is the center of the conic and $\epsilon=1$ if the conic is an ellipse and $\epsilon=-1$ if the conic is a hyperbola. As a consequence, $(p,\0,0)$ belongs to EMH.
\end{lemma}

\begin{proof}
Consider $p_1 \in S_1$ and $ p_2 \in S_2$ with non-parallel tangent planes. By an adequate affine change of variables, we may assume that
$p_1=(0,\0,1)$, $p_2=(0,\0,-1)$ and the mid-plane is $z=0$. We may also assume that the conic is contained in the $xz$-plane and that
the intersection of the tangent planes $T_{p_1}S_1$ and $T_{p_2}S_2$ is the $\y$-space. By lemma \ref{lemma:MixedTerm}, these conditions imply that the coefficients
of $u_1v_{1,j}$ and $u_2v_{2,j}$ are zero.

Now assuming that the center of the conic is the point $(p,\0,0)$,  $S_1$ and $S_2$ are graphs of functions given by equations
\eqref{eq:parameterS1} and \eqref{eq:parameterS2}, for some $\epsilon=\pm 1$. As a consequence, $(p,\0,0)$ satisfies the system \eqref{eq:SystemEPM1}, which implies that $(p,\0,0)$ belongs to EMH.
\end{proof}

From the above two lemmas we can conclude the main result of this section.

\begin{proposition}
The set of centers of conics which make contact of order $\geq 3$ at points $p_i\in S_i$ at directions $Y_i$ which are
$h_i$-orthogonals to the intersection of $T_{p_i}S_i$, $i=1,2$, coincides with the set EMH.
\end{proposition}

\section{Regularity of the EMH}

 In this section, we shall study the regularity of the Envelope of Mid-Hyperplanes.
 
\subsection{A general condition for regularity}

Let $F$ be given by equation \eqref{eq:Normais1} and consider the map
$$
\begin{array}{cccl}
G: & \mathbb{R}^{2N} \times \mathbb{R}^{N+1} & \longrightarrow & \mathbb{R}^{2N+1}  \\
    & (u_1,\v_1,u_2,\v_2,X) & \longmapsto & \left(F,F_{u_1},F_{\v_1},F_{u_2},F_{\v_2}\right).\\
\end{array}
$$
Then the set $EMH$ is the projection in $\mathbb{R}^{N+1}$ of the set $G=0$. If $\0\in\mathbb{R}^{2N+1}$ is a regular value of $G$, then $G^{-1}(0)$ is a $N$-dimensional submanifold
of $\mathbb{R}^{3N+1}$. We want to find conditions under which $\pi_2(G^{-1}(0))$ becomes smooth, where $\pi_2(u_1,\v_1,u_2,\v_2,X)=X$.

The jacobian matrix of $G$ is
\[
JG=\left(\begin{BMAT}[5pt]{c|c}{c|c}
\begin{BMAT}[5pt]{cccc}{c}
 F_{u_1}  & F_{\v_1}  & F_{u_2}    & F_{\v_2}
\end{BMAT}
&\begin{BMAT}{ccc}{c}
 \ \ F_{x}  & \ \  F_{y}  & \ \ F_{z}
\end{BMAT} \\
\begin{BMAT}[10pt]{c}{c}
\begin{array}{cccc}
  F_{u_1u_1} & F_{u_1\v_1} & F_{u_1u_2} & F_{u_1\v_2} \\
  F_{\v_1u_1} & F_{\v_1\v_1} & F_{\v_1u_2} & F_{\v_1\v_2} \\
  F_{u_2u_1} & F_{u_2\v_1} & F_{u_2u_2} & F_{u_2\v_2} \\
  F_{\v_2u_1} & F_{\v_2\v_1} & F_{\v_2u_2} & F_{\v_2\v_2}
\end{array}%
\end{BMAT}
& \begin{BMAT}{ccc}{cccc}
\ \ F_{u_1x} & \ \ F_{u_1y} & \ \ F_{u_1z} \\
\ \ F_{\v_1x} &\ \ F_{\v_1y} & \ \ F_{\v_1z} \\
\ \ F_{u_2x} & \ \ F_{u_2y} & \ \ F_{u_2z} \\
\ \ F_{\v_2x} & \ \ F_{\v_2y} &  \ \ F_{\v_2z} \\
\end{BMAT}
\end{BMAT}
\right).
\]
Denote by $JG1$ the matrix of second derivatives of $F$ with respect to the parameters $u_1,\v_1,u_2,\v_2$, which is the $2N\times 2N$ matrix
consisting of the elements $JG(j_1,j_2)$, $2\leq j_1\leq 2N+1$, $1\leq j_2\leq 2N$.
Denote $\det\left(JG1(u_1,\v_1,u_2,\v_2,X)\right)$ by $\Delta(u_1,\v_1,u_2,\v_2)$.

\begin{thm} \label{teorema_de_suavidadeSC3C}
If $\Delta\neq0$, then the $EMH$ is smooth at the point $X$.
\end{thm}

\begin{proof}
Since the mid-plane is non-degenerate, the equalities $F_x=F_\y=F_z=0$ cannot occur simultaneously. 
Moreover, at points of the envelope, $F_{u_1}=F_{\v_1}=F_{u_2}=F_{\v_2}=0$.
Thus the hypothesis implies that $JG$ has rank $2N+1$ and so $G^{-1}(\0)$ is a regular $N$-submanifold of $\mathbb{R}^{3N+1}$.
Moreover, the hypothesis $\Delta\neq 0$ implies that the differential of $\pi$ restricted to $G^{-1}(\0)$
is an isomorphism. We conclude that the $EMH$ is smooth at this point.
\end{proof}

\subsection{The case of surfaces}

In the case of surfaces (N=2), we can understand better the meaning of the condition $\Delta\neq 0$. 
By lemma \ref{modelolocalsuperficies1}, we may assume that $S_1$ and $S_2$ are graphs of functions $f_1$ and $f_2$ given by 
$$
f_1(u_1,v_1)=1+\epsilon pu_1-\frac{1}{2}(p^2+\epsilon)u_1^2+av_1^2+a_0u_1^3+a_1u_1^2v_1+a_2u_1v_1^2+a_3v_1^3+O(4),
$$
$$
f_2(u_2,v_2)=-1-\epsilon pu_2+\frac{1}{2}(p^2+\epsilon)u_2^2+bv_2^2+b_0u_2^3+b_1u_2^2v_2+b_2u_2v_2^2+b_3v_2^3+O(4).
$$

Long but straightforward calculations using formula \eqref{eq:Normais1} show that the jacobian matrix $JG1$ at point $(0,0,0,0,p,0,0)$ is given by
\[\footnotesize{
\left(\begin{BMAT}{cccc}{cccc}
3p^2+3p^4-6pa_{0} & -2pa_{1}               & 0                & 0 \\
-2pa_{1}          & -2ap^2-2a_{2}p       & 0                & (a+b)(p^2+1)\\
0               & 0                    & -3p^4-3p^2-6pb_{0} &  -2pb_{1} \\
0               & (a+b)(p^2+1)     & -2pb_{1}           &  -2bp^2-2b_{2}p
\end{BMAT}
\right).}
\]
Thus the condition $\Delta\neq 0$ means that the determinant of this $4\times 4$ matrix is not zero.
We give below geometric interpretations of this condition in some particular cases, but a geometric 
interpretation in the general case remains to be given.

\subsubsection{$3+3$ contact with quadrics}

The condition $a+b=0$ is equivalent to the existence of a quadric $Q$ with contact of order $\geq 3$ with $S_1$ at $p_1$ and $S_2$
at $p_2$. In this case, the condition $\Delta\neq 0$ can be written as $\delta_1\delta_2\neq 0$, where 
\[
\delta_1=\det\left[
\begin{array}{cc}
3p^2+3p^4-6pa_{0} & -2pa_{1}     \\
-2pa_{1}          & -2ap^2-2a_{3}p 
\end{array}
\right] ,
\]
\[
\delta_2=\det\left[
\begin{array}{cc}
-3p^2-3p^4-6pb_{0} & -2pb_{1}     \\
-2pb_{1}          & -2bp^2-2b_{3}p 
\end{array}
\right] .
\]

The condition $\delta_1=0$ means that the quadric $Q$  has in fact a higher order contact with $S_1$ at $p_1$. In fact 
taking 
$$
h_1(x,y,z)=(x-p)^2+z^2+ay^2,
$$
consider the contact function $\bar{h}_1(u_1,v_1)=h_1(\psi_i(u_1,v_1))$.  Then $\bar{h}_1(0,0)=0$ and $(0,0)$ is a critical point of $\bar{h}_1$. 
One can verify that $\delta_1=0$ if and only if $(0,0)$ is a degenerate critical point of $\bar{h}_1$. A similar geometric interpretation holds for 
$\delta_2$.

\subsubsection{The case $a_1=a_{2}=b_{1}=b_{2}=0$.}

In this case we have
$$
\Delta=\left( 3p^2+3p^4-6pa_{0} \right)\cdot \left(-3p^4-3p^2-6pb_{0}\right) \delta, 
$$
where $\delta=-\left( (a-b)^2p^4+(a+b)^2+2p^2(a+b)^2   \right) <0$. Then the condition $\Delta\neq 0$ says that the contacts of the conic with $S_1$ at $p_1$ and $S_2$ at $p_2$ are both exactly $3$.

\section{ A counter-example for the reflection property}

In \cite{Sapiro}, it is proved that if the AESS of a pair of planar curves is contained in a line, then there exists an affine reflection taking one curve into the other. This fact is not true for the EMH
of a pair of hypersurfaces as the following example shows us.

Consider $\gamma_1(t)=(t,0,f(t))$ a smooth convex curve and let $\gamma_2(t)=(t-\lambda f(t),0,-f(t))$, $\lambda\in\mathbb{R}$, be obtained from $\gamma_1$ by an affine reflection. Let $S_1$
and $S_2$ be rotational surfaces obtained by rotating $\gamma_1$ and $\gamma_2$ around the $z$-axis. $S_1$ and $S_2$ can be parameterized by
$$
\phi_1(t,\theta)=\left(  t\cos(\theta),t\sin(\theta),f(t)   \right)
$$
and
$$
\phi_2(t,\theta)=\left(  (t-\lambda f(t))\cos(\theta), (t-\lambda f(t))\sin(\theta), -f(t) \right).
$$
The intersection $Z$ of the tangent planes at $\phi_1(t,\theta)$ and $\phi_2(t,\theta)$ has direction $\left( \sin(\theta), -\cos(\theta), 0 \right)$.

Observe that the vectors $Y_1=(\phi_1)_t$ and $Y_2=(\phi_2)_t$ are orthogonal to $Z$ in the Blaschke metric. This implies that the EMH of this pair
of surfaces is contained in the plane $z=0$. But it is clear that $S_2$ is not an affine reflection of $S_1$.

\bibliographystyle{amsplain}

\end{document}